\documentclass[a4paper]{article}
\usepackage[utf8]{inputenc}
\usepackage[T1]{fontenc}

\usepackage{lmodern}
\usepackage{array}
\usepackage{graphicx}
\usepackage{subfig}
\usepackage{amssymb, amsfonts, amsthm, amsmath,bbm}
\usepackage{enumerate}
\usepackage{comment}

\usepackage[english]{babel}

\usepackage[colorlinks]{hyperref}

\usepackage{tikz}

\numberwithin{equation}{section}

\theoremstyle{plain}
\newtheorem{theorem}{Theorem}[section]

\newtheorem{lemma}[theorem]{Lemma}
\newtheorem{fact}[theorem]{Fact}
\newtheorem{proposition}[theorem]{Proposition}

\theoremstyle{definition}

\theoremstyle{remark}
\newtheorem{remark}[theorem]{Remark}

\newcommand{\N}{\mathbb{N}}
\newcommand{\Z}{\mathbb{Z}}
\newcommand{\R}{\mathbb{R}}

\newcommand{\ind}[1]{\mathbbm{1}_{\left\{#1\right\}}}

\newcommand{\floor}[1]{{\left\lfloor #1 \right\rfloor}}
\newcommand{\ceil}[1]{{\left\lceil #1 \right\rceil}}

\DeclareMathOperator{\E}{\mathbf{E}}
\renewcommand{\P}{\mathbf{P}}

\newcommand{\crochet}[1]{{\langle #1 \rangle}}
\newcommand{\e}{\mathrm{e}}

\newcommand{\dd}{\mathrm{d}}

\renewcommand{\bar}[1]{\overline{#1}}
\renewcommand{\tilde}[1]{\widetilde{#1}}

\renewcommand{\epsilon}{\varepsilon}
\renewcommand{\phi}{\varphi}
\newcommand{\T}{\mathbbm{T}}

\title{The longest branches in a non-Markovian phylogenetic tree}
\author{Sergey Bocharov\thanks{Department of Foundational Mathematics, Xian Jiaotong-Liverpool University, Ren’Ai Road
111, Suzhou 215123, China, e-mail: \texttt{Sergey.Bocharov@xjtlu.edu.cn}.} \and Simon C. Harris\thanks{Department of Statistics, University of Auckland, 38 Princes Street, Auckland, 1001, New Zealand, e-mail: \texttt{simon.harris@auckland.ac.nz}} \and Bastien Mallein\thanks{Institut de Mathématiques de Toulouse, UMR 5219, Université de Toulouse, UPS, F-31062 Toulouse Cedex 9, France, email: \texttt{bastien.mallein@math.univ-toulouse.fr}}}
\date{\today}

\begin{document}
\maketitle

\begin{abstract}
Consider a Bellman--Harris-type branching process, in which individuals evolve independently of one another, giving birth after a random time $T$ to a random number $L$ of children. In this article, we study the asymptotic behaviour of the length of the longest branches of this branching process at time $t$, both pendant branches (corresponding to individuals still alive at time $t$) and interior branches (corresponding to individuals dead before time $t$).
\end{abstract}

\section{Introduction}

We consider a branching particle system initiated by a single particle, defined in the following way. Let $(T,L)$ be a random element of $\R_+ \times \Z_+$. Each individual in that process evolves independently of one another. An individual $u$ stays alive for $T_u$ units of time, before giving birth to $L_u$ children, where $(T_u,L_u)$ is an independent copy of $(T,L)$. Those children reproduce independently with the same law of their parent. We work under the conditions
\begin{equation}
\label{eqn:supercritical}
   \P(T = 0) = \P(T = \infty) = 0 \quad \text{and} \quad \E(L) \in (1,\infty).
\end{equation}
The first condition avoids considering a degenerate situation in which an individual immediately dies out or live forever, while the second one guarantees that the population is supercritical. As a result, the process survives forever with positive probability.

When $T$ and $L$ are independent, this process was introduced by Bellman and Harris \cite{BeH} as a generalisation of continuous-time Galton--Watson processes (in which $T$ is exponentially distributed). This model, in which the lifetime of an individual is correlated with its number of offspring is often called a \emph{Sevast'yanov process}, after their namesake who introduced such processes in \cite{Sevastyanov,Sevastyanov2}. An other way to see Sevast'yanov processes is as a subclass of Crump--Mode--Jagers (CMJ) branching processes in which individuals can only give birth to children at their death-time.

Let us remark that the condition $T > 0$ a.s. could be slightly weakened. In particular, if we identify as children of an individual all its descendants born at its death time, we recover a branching process satisfying \eqref{eqn:supercritical}, provided that $\E(L;T = 0) \leq 1$ (so that the clusters of individuals born at the same time as their parent form a sub-critical branching processes, hence have finite total progeny). If $\E(L; T = 0) > 1$, then the process would explode in finite time a.s. on its survival event, which justifies its restriction.

As is usual, we can identify individuals in a Sevast'yanov process with a label that also encodes their genealogy (see Section \ref{sec:pf} for more details).
We then denote by $\T$ be the set of all individuals that ever live. The individual $u \in \mathbbm{T}$ is born at time $b_u$, and dies at time $d_u$.
We also let $\T_t =\{u \in \T : b_u \leq t\}$ be the set of individuals born before time $t$, and  $\mathcal{N}_t = \{u \in \T : b_u \leq t < d_u\}$ the set of individuals alive at time $t$.

The objective of this article is to study the asymptotic behaviour of the ages of the oldest living and extinct individuals at time $t$ in this population model. This result extends the recent work of Bocharov and Harris \cite{BoHa}, that proved a similar result for continuous-time Galton--Watson trees. In particular, we aim to describe the asymptotic behaviour of
\begin{equation}
  \label{eqn:longestEdges}
  M_t^\mathrm{p} = \max_{u \in \mathcal{N}_t} (t- b_u) \quad \text{ and } \quad M_t^\mathrm{i} = \max_{u \in \T_t\setminus\mathcal{N}_t} (d_u-b_u),
\end{equation}
which are, respectively, the age of the oldest individual still alive at time $t$, and the age of the oldest individual that died before time $t$.
That is, $M_t^\mathrm{p}$ corresponds to the longest \emph{pendant} edge of the genealogical tree up to time $t$, and $M_t^\mathrm{i}$ is the longest \emph{interior} edge up to time $t$.
We will define below a deterministic function $(\ell_t)$ such that $(M_t^\mathrm{p}-\ell_t, M_t^\mathrm{i} - \ell_t)$ jointly converge in distribution as $t \to \infty$.

Whenever it exists, the so-called \emph{Malthusian parameter} plays an important role in the study of CMJ branching processes. For a Sevast'yanov process, it is defined as the unique parameter $\alpha > 0$ satisfying
\begin{equation}
  \label{eqn:malthusianParameter}
  \E(L e^{-\alpha T}) = 1.
\end{equation}
Note that \eqref{eqn:supercritical} guarantees the existence of this parameter as $\alpha \mapsto \E(Le^{-\alpha T})$ is continuous and strictly decreasing on $[0,\infty)$, starting from $\E(L)>1$ down to $0$.

The Malthusian parameter allows us to study the rate of exponential growth of the population over time. Provided that
\begin{equation}
  \label{eqn:nonLattice}
  \forall a > 0,\ \ \P(T \in a \N) < 1,
\end{equation}
i.e. that the law of $T$ is non-lattice, Nerman \cite{Nerman} proved that the size of the population grows exponentially at rate $\alpha$. Specifically, in our situation, we have
\begin{equation}
  \label{eqn:exponentialGrowth}
  \lim_{t \to \infty} e^{-\alpha t} \#\mathcal{N}_t =  \frac{1 - \E(e^{-\alpha T})}{\E(\alpha T e^{-\alpha T} L)} Z_\infty \quad \text{ in probability,}
\end{equation}
as soon as
\begin{equation}
 \label{eqn:llogl}
 \E(e^{-\alpha T} L \log_+ L) < \infty.
\end{equation}
Since $\sup_{t \in \R_+} t e^{-\alpha t} < \infty$, we remark that $\E(\alpha T e^{-\alpha T} L) < \infty$ by \eqref{eqn:supercritical}.

The random variable $Z_\infty$ is the a.s. limit of the so-called \emph{Malthusian martingale} defined as
\begin{equation}
  \label{eqn:malthusianMartingale}
  Z_t := \sum_{u \in \mathcal{N}_t} L_u e^{-\alpha d_u} = \sum_{v \in \T} e^{-\alpha b_v} \ind{b_v > t \geq b_{\pi v}},
\end{equation}
where $\pi v$ represents the parent of $v$. Here, $(Z_t, t \geq 0)$ is a non-negative martingale with respect to the filtration $\mathcal{F}^\uparrow$ where $\mathcal{F}^\uparrow_t := \sigma\left((T_u, L_u), u \in \mathbbm{T} : b_u < t\right)$. This is often called the filtration of the coming generation, as information on the death time and number of children of individuals alive at time $t$ is already available at that time, thereby looking a generation into the future. This \emph{coming generation} at time $t$ consists of all individuals who are born after $t$ but whose parents were born before $t$. For more details, see the classic book of Jagers \cite{Jagers}.

The exponential growth of branching population models have a long history. Kesten and Stigum \cite{KesSti} proved that when $\E(L \log L) < \infty$, a discrete time Galton-Watson process\footnote{That is, $T= 1$ a.s.} $(\#\mathcal{N}_n)$ grows at the same rate as $\E(\#\mathcal{N}_n)$. This result was then extended to continuous-time Galton-Watson processes\footnote{That is, $T$ is exponentially distributed.} by Athreya and Karlin \cite{AthKar}. Athreya \cite{Athreya} then showed that for a Bellman--Harris process, the convergence in \eqref{eqn:exponentialGrowth} holds in distribution, identifying the law of $Z_\infty$ as a fixed point of the smoothing transform. This result was extended by Doney \cite{Doney} to general CMJ branching processes. Next, Athreya and Kaplan \cite{AtK} proved the stronger result \eqref{eqn:exponentialGrowth} for Bellman--Harris processes, and Nerman \cite{Nerman} extended it to CMJ processes, under the optimal integrability conditions.

Refinements of these results have been made available under more generality. Indeed, observe that the process $(b_u, u \in \mathbbm{T})$ can be thought of as a branching random walk with non-decreasing paths. Additive martingales of branching random walks have been studied since at least \cite{Watanabe}, with optimal integrability conditions given by Biggins \cite{Biggins} (see also Lyons \cite{Lyons} for a simple proof) and Alsmeyer and Iksanov \cite{AlI}. The martingale $(Z_t)$ can then be seen as an additive martingale estimated along a stopping line, uniform integrability of the martingale along genealogical lines therefore extends along stopping lines \cite{BiK04}.

To study the asymptotic behaviour of the length of the longest edges in the Sevast'yanov process, we impose a regularity assumption on the tail of law of the lifetime $T$ of individuals. We assume that $x \mapsto \P(T >  \log x)$ is a regularly varying function at $\infty$. This assumption can be restated in a clearer fashion as follows: there exists $\beta \geq 0$ such that for all $a \in \R_+$, we have
\begin{equation}
  \label{eqn:regularTail}
\lim_{t \to \infty} \P(T > t+a|T > t) = \lim_{t \to \infty} \frac{\P(T > t + a)}{\P(T > t)}= e^{-\beta a}.
\end{equation}
Let us underscore that this condition includes lifetimes with heavy tail distributions. In particular, if $\P(T > t) \approx t^{-c}$, then \eqref{eqn:regularTail} holds with $\beta = 0$.

The characteristic length $\ell_t$ of the longest edges at time $t$ is defined by the formula
\begin{equation}
  \label{eqn:defLt}
  \ell_t = \inf\{ \ell > 0 : e^{-\alpha \ell} \P(T > \ell) \leq e^{-\alpha t}\}.
\end{equation}
The intuition for this being a good characteristic length choice is as follows. First observe that, due to the exponential growth of the process, branches of large length $\ell$ are most likely to be born around time $t - \ell$. Around this time, by \eqref{eqn:exponentialGrowth}, there are around $e^{\alpha (t-\ell)}$ newborn individuals, each of which having probability $\P(T > \ell)$ to create such a long edge. The characteristic length $\ell_t$ is then chosen in such a way that on average there are $O(1)$ such long branches of length $\ell_t$ at time $t$, whilst in contrast, we do not expect to find any branch significantly longer than the characteristic length $\ell_t$ at time $t$.
We prove in forthcoming Lemma~\ref{lem:linearGrowth} that $\ell_t$ grows linearly with $t$, specifically that
\begin{equation}
  \label{eqn:linearGrowthL}
  \lim_{t \to \infty} \frac{\ell_t}{t} = \frac{\alpha}{\alpha+\beta}.
\end{equation}

Our first main result is that the a.s. growth rate of the longest edges in $\T_t$ is identical to the growth rate of $\ell_t$ as $t \to \infty$.
\begin{theorem}
\label{thm:ps}
Under assumptions \eqref{eqn:supercritical}, \eqref{eqn:llogl} and \eqref{eqn:regularTail}, we have
\[
  \lim_{t \to \infty} \frac{M_t^\mathrm{p}}{t} =  \frac{\alpha}{\alpha + \beta} \quad \text{a.s. on the survival event}.
\]
In addition, if $\beta > 0$ then
\[
  \lim_{t \to \infty} \frac{M_t^\mathrm{i}}{t} = \frac{\alpha}{\alpha + \beta} \quad \text{a.s. on the survival event}.
\]
\end{theorem}

We believe that in a typical heavy tail scenario with $\beta = 0$, we would still have $\lim_{t \to \infty} \frac{M_t^\mathrm{i}}{t} = 1$ a.s. However, our method of proof would need to be adapted to treat this particular case. We postpone to Remark~\ref{rem:heavyTail} the discussion related to this particular case.

To explore in more details the length of the longest edges in $\T_t$ we now define the following point processes on $\R$,
\begin{equation}
  \label{eqn:defPP}
  \mathcal{E}^\mathrm{p}_t = \sum_{u \in \mathcal{N}_t} \delta_{(t-b_u) - \ell_t} \quad \text{ and } \quad \mathcal{E}^{\mathrm{i}}_t = \sum_{u \in \T_t\setminus\mathcal{N}_t} \delta_{(d_u-b_u) - \ell_t}.
\end{equation}
The point process $\mathcal{E}^\mathrm{p}_t$ allows us to study the length of all pendant edges relative to length $\ell_t$, while $\mathcal{E}^\mathrm{i}_t$ gives a convenient way to encode the length of interior edges compared to length $\ell_t$. In particular, we observe that the right-most atom in $\mathcal{E}^\mathrm{p}_t$ is located at $M_t^\mathrm{p} - \ell_t$, and the right-most atom in $\mathcal{E}^\mathrm{i}_t$ corresponds to $M_t^\mathrm{i} - \ell_t$.

The main result of the article can be expressed as follows.
\begin{theorem}
\label{thm:main}
Under assumptions \eqref{eqn:supercritical}, \eqref{eqn:llogl} and \eqref{eqn:regularTail}, $(\mathcal{E}^\mathrm{p}_t,\mathcal{E}^\mathrm{i}_t)$ jointly converge in distribution, for the topology of vague convergence of point processes, to $(\mathcal{E}^\mathrm{p}_\infty,\mathcal{E}^\mathrm{i}_\infty)$ whose law can be described as follows: conditionally on $Z_\infty$, $\mathcal{E}^\mathrm{p}_\infty$ and $\mathcal{E}^\mathrm{i}_\infty$ are independent Poisson point processes with intensity $c_\star Z_\infty \alpha e^{-(\alpha + \beta)x} \dd x$ and $c_\star Z_\infty \beta e^{-(\alpha + \beta)x} \dd x$ on $\R$ respectively, where
\begin{equation}
  \label{eqn:defCstar}
  c_\star = \E(\alpha T e^{-\alpha T} L)^{-1}.
\end{equation}
In addition, we have the convergence in distribution of the longest pendant and interior edges relative to $\ell_t$ given by:
\[
  \lim_{t \to \infty}
  (M_t^\mathrm{p} - \ell_t,M_t^\mathrm{i} - \ell_t)  =  (\max \mathcal{E}^\mathrm{p}_\infty,\max \mathcal{E}^\mathrm{i}_\infty),
\]
with the convention that the largest element of the empty point measure is $-\infty$.
\end{theorem}

\begin{remark}
We mention that proving the convergence of a family $(\mathcal{E}_t)$ of point processes on $\R$ towards $\mathcal{E}_\infty$ in distribution for the topology of vague convergence, as well as the convergence in distribution of $\max \mathcal{E}_t$ to $\max \mathcal{E}_\infty$, is equivalent to proving, for all $k \in \N$, the joint convergence in distribution of the positions of the first $k$ atoms of $\mathcal{E}_t$ to the position of the first $k$ atoms of $\mathcal{E}_\infty$. We refer to \cite[Lemma~4.4]{BBCM} for more details.
\end{remark}

\begin{remark}
By immediate Poisson computation, when considering the superposition of (the conditionally on $Z_\infty$) independent Poisson point processes $\mathcal{E}^\mathrm{i}_\infty$ and $\mathcal{E}^\mathrm{p}_\infty$, corresponding to the limiting point measure of \emph{all} edges of length around $\ell_t$ in $\T_t$, each atom of the combined point process will belong to $\mathcal{E}^\mathrm{p}_\infty$ with probability $\frac{\alpha}{\alpha + \beta}$, independently of its position. In particular, the longest edge in $\T_t$ will be pendant with probability converging to $\frac{\alpha}{\alpha + \beta}$ as $t \to \infty$.
\end{remark}

\begin{remark}
\label{rem:heavyTail}
Observe that in the previous theorem, if $\beta = 0$ then $\mathcal{E}^\mathrm{i}_t$ converges in law to the null point measure, and $\lim_{t \to \infty} M^\mathrm{p}_t- M^\mathrm{i}_t = \infty$ in probability. This is due to the fact that this situation corresponds to individual lifetimes with a tail that decays at a sub-exponential rate. Therefore, conditionally on $T > x$, one expects that $T \gg x$ with high probability, and the large majority of long edges in the process therefore correspond to individuals that are still alive at time $t$. Introducing more precise assumptions on the tail decay of $T$, such as $\P(T > t)$ is regularly varying at $\infty$, one could define, using similar methods as the ones we develop here, a function $\bar{\ell}_t$ such that $M_t^\mathrm{i}-\bar{\ell}_t$ converges in distribution towards a non-degenerate limit.
\end{remark}

We prove our main results in the next section.

\section{Proof of the main theorems}
\label{sec:pf}

This section is divided into three parts. We first provide in Section~\ref{subsec:UHconstruction} an explicit construction of the Sevast'yanov process using the Ulam--Harris--Neveu notation. We then give in Section~\ref{subsec:characLength} some properties on the characteristic length $\ell_t$ of the longest branches in $\T_t$. Finally, Section~\ref{subsec:pf} is dedicated to prove, first Theorem~\ref{thm:main}, then Theorem~\ref{thm:ps}.

\subsection{Construction of the Sevast'yanov process}
\label{subsec:UHconstruction}
We introduce here some of the notation and setup for our model.
The Sevast'yanov process can classically be constructed using the Ulam--Harris--Neveu notation as follows. Let $\mathcal{U} = \cup_{n \geq 0} \N^n$ the set of finite words over the alphabet $\N$. We let $\{(T_u,L_u), u \in \mathcal{U}\}$ be a family of i.i.d. copies of $(T,L)$. We then define
\[
  \T = \left\{u \in \mathcal{U} : \ \forall\ 1\leq   j \leq |u|,\  u(j) \leq L_{u_{j-1}}\right\},
\]
where given $u = (u(1),\ldots, u(n)) \in \N^n$, we write $|u|=n$ the length of $u$,  $u_k = (u(1),\ldots, u(k))$ the prefix consisting of the first $k \leq n$ letters of $u$, and $u_0=\emptyset$ with $|u_0|=0$. For $u \in \mathbbm{T}$, we write
\[
  b_u = \sum_{0 \leq j < |u|} T_{u_j}, \quad \text{and} \quad  d_u = b_u + T_u = \sum_{0 \leq j \leq |u|} T_{u_j}.
\]
In words, $\emptyset \in \N^0$ corresponds to the initial ancestor of the population, whose genealogical tree is given by $\mathbbm{T}$. The individual $u \in \mathbbm{T}$ is born at time $b_u$, and dies at time $d_u$.
An individual $u \in \N^n$ is from the $n$th generation and is identified as the $u(n)$th child of the $u(n-1)$th child of ... of the $u(1)$th child of the initial ancestor of the population.

For all $t \geq 0$, we write
\[
  \T_t = \{u \in \T : b_u \leq t\} \quad \text{and} \quad \mathcal{N}_t = \{u \in \T : b_u \leq t < d_u\}
\]
which are, respectively, the set of individuals born before time $t$ and the set of individuals alive at time $t$. The length of the branch associated to individual $u \in \T_t$ is then defined as $\min(t,d_u) - b_u$, which is consistent with the point processes defined in \eqref{eqn:defPP}.

\subsection{Properties of the characteristic length of the longest edges}
\label{subsec:characLength}

The definition of the characteristic length $\ell_t$ in \eqref{eqn:defLt} can be explained informally as follows. The population in the Sevast'yanov process growing at exponential rate \ref{eqn:exponentialGrowth}, the large majority of branches of length $\ell$ come from individuals that were born around time $t-\ell$. Since there are approximately $e^{\alpha (t - \ell)}$ individuals born around this time, and that each of them have probability $\P(T > \ell)$ to give birth to a branch of length at least $\ell$, we conclude that there will be branches of length $\ell$ in $\T_t$ if $e^{\alpha(t-\ell)} \gg 1$, while no branches will exist if $e^{\alpha(t-\ell)}\P(T > \ell) \ll 1$.

We show in the next lemma that the function $\ell_t$ defined above corresponds to the critical situation.
\begin{lemma}
\label{lem:asymptoticLt}
If $\P(T>0) =1$ and \eqref{eqn:regularTail} holds, then the function $\ell$ defined in \eqref{eqn:defLt} verifies
\[
  \lim_{t \to \infty} e^{\alpha(t-\ell_t)} \P(T > \ell_t) = 1.
\]
\end{lemma}

\begin{proof}
We observe that $f : \ell \mapsto e^{-\alpha \ell} \P(T > \ell)$ is a c\`adl\`ag decreasing function on $[0,\infty)$ with $f(0)=1$ and $\lim_{\ell \to \infty} f(\ell) = 0$. Writing $f^{-1}$ its right-continuous generalised inverse, we have
\[
  \ell_t = f^{-1}(e^{-\alpha t}) \text{ for all } t > 0.
\]

In particular, we observe that
\[
  e^{-\alpha \ell_t} \P(T > \ell_t) \leq e^{-\alpha t} \leq e^{-\alpha \ell_t} \P(T \geq \ell_t).
\]
Using \eqref{eqn:regularTail}, we have
\begin{equation*}
  \lim_{t \to \infty} \frac{\P(T \geq t)}{\P(T > t)} = 1,
\end{equation*}
which allows us to complete the proof.
\end{proof}

We now remark that under assumption \eqref{eqn:regularTail}, $\ell_t$ grows at a linear rate.
\begin{lemma}
\label{lem:linearGrowth}
Under assumption \eqref{eqn:regularTail}, we have $\lim_{t \to \infty} \frac{\ell_t}{t} = \frac{\alpha}{\alpha + \beta}$.
\end{lemma}

\begin{proof}
We write $\phi(t) = \exp\left(\ell_{\alpha^{-1} \log t}\right)$, we observe that
\[
  \phi(t) = \inf\{x > 0 : x^{-\alpha} \P(T > \log x) \leq 1/t \},
\]
therefore $\phi$ is the right-continuous inverse of $x \mapsto \frac{x^{\alpha}}{\P(T > \log x)}$, which is a regularly varying function at $\infty$ with index $\alpha + \beta$ (see e.g. \cite[Theorem 1.5.12]{BGT}). As a result, we deduce that $\phi(t)$ is regularly varying at $\infty$ with index $\frac{1}{\alpha + \beta} > 0$. In particular, this yields
\[
  \frac{1}{\alpha + \beta} = \lim_{t \to \infty} \frac{\log \phi(t)}{\log t} = \lim_{t \to \infty} \frac{\ell_{\alpha^{-1}\log t}}{\log t},
\]
which completes the proof.
\end{proof}

\begin{remark}
\label{rem:expTail}
As a slight refinement of Lemma~\ref{lem:linearGrowth}, let us mention that if $T$ has an exponential tail, i.e. there exist $c,\beta > 0$ such that $\P(T > t) \sim c e^{-\beta t}$, then $\ell_t = \frac{\alpha t + \log c}{\alpha + \beta} + o(1)$ as $t \to \infty$. In particular, if $T$ and $L$ are independent with $T$ exponentially distributed with parameter $\beta$ , then $(\mathcal{N}_t, t \geq 0)$ is a continuous-time Galton-Watson tree. In this case, the Malthusian exponent is given by $\beta (\E(L) - 1)$, and we obtain that the length of the longest edges in that tree are of order
\[
  \ell_t = \left( 1 - \frac{1}{\E(L)}\right) t + o(t) \quad \text{as $t \to \infty$,}
\]
in accordance with the results of Bocharov and Harris \cite{BoHa}.
\end{remark}

\subsection{Proof of the main theorems}
\label{subsec:pf}

We prove in this section Theorems~\ref{thm:ps} and \ref{thm:main}. We begin by the proof of this second theorem, that will be used to the first one as a consequence.

The proof of Theorem~\ref{thm:main} relies on the study of the joint asymptotic behaviour of $\mathcal{E}_t^\mathrm{p}$ and ${\mathcal{E}}_t^\mathrm{i}$, the latter of which we approximate by
\[
  \tilde{\mathcal{E}}_t^\mathrm{i} :=
   \sum_{u \in \T_t\setminus\mathcal{N}_t} \delta_{(d_u-b_u) - \ell_t}\,\ind{b_u > t-3\ell_t/2}
    = \sum_{u \in \mathbbm{T}} \ind{d_u < t,\,b_u > t-3\ell_t/2}\delta_{(d_u - b_u) - \ell_t}.
\]
This modified extremal process counts the long interal edges born after time $t - 3 \ell_t/2$. As a result, we can guarantee that each edge counted in $\mathcal{E}_t^\mathrm{p}(\R_+)$ or $\tilde{\mathcal{E}}_t^{\mathrm{i}}(\R_+)$ exists at time $t - \ell_t$, which will allow us to use the spine decomposition at that time.

The main result of the section is the following convergence of the joint Laplace transform of $\mathcal{E}^\mathrm{p}$ and $\tilde{\mathcal{E}}^\mathrm{i}$, proving their convergence to randomly shifted Poisson point processes.

\begin{proposition}
\label{prop:cv}
Let $A > 0$ and $\phi,\psi$ two continuous non-negative bounded functions with support in $[-A,\infty)$, we have
\begin{multline*}
  \lim_{t \to \infty} \E\left( \exp\left( -\crochet{\mathcal{E}^\mathrm{p}_t,\phi} - \crochet{\tilde{\mathcal{E}}^\mathrm{i}_t, \psi} \right)\right)\\
  = \E\left( \exp\left( - c_\star Z_\infty \int_{\R} \left\{ \alpha (1 - e^{-\phi(x)}) + \beta (1 - e^{-\psi(x)})\right\}e^{-(\alpha + \beta)x} \dd x \right) \right).
\end{multline*}
\end{proposition}

Before turning to the proof of this proposition, we first show that $\tilde{\mathcal{E}}^\mathrm{i}_t$ is a good approximation of $\mathcal{E}^\mathrm{i}_t$, more specifically that ${\mathcal{E}}^\mathrm{i}_t - \tilde{\mathcal{E}}^\mathrm{i}_t \to 0$ in probability for the topology of vague convergence.

\begin{lemma}
\label{lem:goodApproximation}
For all $A > 0$, we have
\[
  \lim_{t \to \infty} \mathcal{E}^\mathrm{i}_t([-A,\infty)) - \tilde{\mathcal{E}}^\mathrm{i}_t([-A,\infty)) = 0
\]
in probability.
\end{lemma}

\begin{proof}
We first remark that $\mathcal{E}^\mathrm{i}_t([-A,\infty)) - \tilde{\mathcal{E}}^\mathrm{i}_t([-A,\infty)) \geq 0$ a.s. Therefore, to complete the proof, we only have to show that
\[
  \lim_{t \to \infty} \E(\mathcal{E}^\mathrm{i}_t([-A,\infty)) - \tilde{\mathcal{E}}^\mathrm{i}_t([-A,\infty))) = 0,
\]
using the Markov inequality.

Using that $T_u = d_u-b_u$ is independent of $b_u$, we observe that
\begin{align*}
  \E(\mathcal{E}^\mathrm{i}_t([-A,\infty)) - \tilde{\mathcal{E}}^\mathrm{i}_t([-A,\infty)))
  &= \E\left( \sum_{u \in \mathbbm{T}} \ind{d_u - b_u - \ell_t \geq -A, \, b_u \leq t - 3\ell_t/2}  \right)\\
  &= \E\left( \sum_{u \in \mathbbm{T}} \ind{T_u \geq \ell_t -A, \, b_u \leq t - 3\ell_t/2}  \right)\\
  &= \E\left( \# \mathbbm{T}_{t-3\ell_t/2} \right) \P(T \geq \ell_t - A).
\end{align*}
By \cite[Proposition 2.2]{Nerman} with $\phi \equiv 1$, we observe that
\begin{equation}
  \label{eqn:typicalGrowthRate}
  \lim_{t \to \infty} e^{-\alpha t} \E(\#\mathbbm{T}_{t}) = c_\star,
\end{equation}
with $c_\star$ the constant defined in \eqref{eqn:defCstar}. Therefore, there exists $C > 0$ such that for all $t \geq 0$, we have
\[
  \E(\mathcal{E}^\mathrm{i}_t([-A,\infty)) - \tilde{\mathcal{E}}^\mathrm{i}_t([-A,\infty)))
  \leq C e^{\alpha (t - 3\ell_t/2)} \P(T \geq \ell_t - A),
\]
which converges to $0$ as $t \to \infty$ since $\P(T > \ell_t - A) \sim e^{\beta A}e^{-\alpha (t-\ell_t)}$.
\end{proof}

In order to estimate the asymptotic behaviour of the joint Laplace transform of $\mathcal{E}_t^\mathrm{p}$ and $\tilde{\mathcal{E}}_t^\mathrm{i}$, we rely on the convergence of general branching processes counted with their characteristics \cite[Section 6.9]{Jagers}. In particular, we will use the following result, which is an adaptation to our settings of \cite[Theorem 3.1]{Nerman}.
\begin{fact}
\label{fct:Nerman}
Let $\phi$ be an c\`adl\`ag function $\R_+\to \R_+$, we write
\[
  Z_t^\phi = \sum_{u \in \mathcal{N}_t} \phi(t - b_u).
\]
Under assumptions \eqref{eqn:supercritical}, \eqref{eqn:llogl} and \eqref{eqn:nonLattice}, we have
\[
  \lim_{t \to \infty} e^{-\alpha t} Z_t^\phi =  m^\phi Z_\infty \text{ in probability,}
\]
where $m^\phi = c_\star \int_0^\infty \alpha e^{-\alpha t} \phi(t) \P(T > t) \dd t$ and $c_\star$ is the constant defined in \eqref{eqn:defCstar}. If $m^\phi = \infty$, the equality remains valid using the convention that $m^\phi Z_\infty$ is zero if $Z_\infty =0$ and infinite otherwise.
\end{fact}

We now have introduced all the necessary tools to prove Proposition \ref{prop:cv}.

\begin{proof}[Proof of Proposition~\ref{prop:cv}]
For all $t > 0$, we write $c_t = t - \ell_t$. Observe that any individual contributing to $\crochet{\mathcal{E}^\mathrm{p}_t,\phi}$ has to satisfy $t - b_u - \ell_t \geq -A$, thus being born before time $c_t+A$ and still alive at time $t$. Similarly, individuals contributing to $\crochet{\tilde{\mathcal{E}}^\mathrm{i}_t,\psi}$ must have a lifetime longer than $\ell_t- A$ while being born after time $t - 3 \ell_t/2$. Therefore, for $t$ large enough, all these individuals must be alive at time $c_t + A$.

Let $\mathcal{F}_t = \sigma(b_u, u \in \mathbbm{T}_t)$ the natural filtration associated to the Sevast'yanov process. Under this filtration, information on the birth time and genealogical relationships of individuals alive at time $t$ are available, as well as the fact that their death time occurs after time $t$. Using the branching property, at time $c_t+A$, we observe that
\begin{align*}
  &\E\left( \exp\left( -\crochet{\mathcal{E}^\mathrm{p}_t,\phi} -\crochet{\tilde{\mathcal{E}}^\mathrm{i}_t, \psi} \right)\middle| \mathcal{F}_{c_t+A}\right)\\
  = &\prod_{u \in \mathcal{N}_{c_t+A}} \E\left( e^{-\phi(c_t- b_u)}\ind{T_u > t-b_u} +e^{-\psi(T_u-\ell_t)}\ind{T_u \leq t-b_u} \middle| \mathcal{F}_{c_t+A} \right) \\
  = &\prod_{u \in \mathcal{N}_{c_t+A}} \left( 1 - g(c_t+A-b_u) \right),
\end{align*}
where, for $x \geq 0$, we write
\begin{multline*}
  g(x) = (1- e^{-\phi(x-A)})\P\left(T > t-(c_t+A-x)\middle| T > x \right)\\
   + \E\left((1 - e^{-\psi(T-\ell_t)}) \ind{T \leq t- (c_t+A-x)} \middle| T > x\right).
\end{multline*}

We first use \eqref{eqn:regularTail} to obtain that
\begin{multline*}
  (1- e^{-\phi(x-A)})\P\left(T > t-(c_t+A-x)\middle| T > x \right)\\
  \sim_{t \to \infty} (1-e^{-\phi(x-A)}) e^{-\beta(x-A)} \frac{\P(T > \ell_t)}{\P(T > x)}.
\end{multline*}

Moreover, using again that for all $z \in \R$,
\[
  \lim_{t \to \infty} \frac{\P(T > \ell_t + z)}{\P(T > \ell_t)} = e^{-\beta z},
\]
we deduce that $\frac{\P(T \in \ell_t + \cdot)}{\P(T > \ell_t)}$ converges vaguely to the distribution $\beta e^{-\beta x} \dd x$ on $\R$. Then, as $1 - e^{-\psi}$ is continuous, bounded, and supported on $[-A,\infty)$, we have
\begin{multline*}
  \E\left((1 - e^{-\psi(T-\ell_t)}) \ind{T \leq t- (c_t+A-x)} \middle| T > x\right)\\
  \sim_{t \to \infty} \int_{-A}^{x-A} (1 - e^{-\psi(y)}) \beta e^{-\beta y} \dd y\frac{\P(T > \ell_t)}{\P(T > x)}.
\end{multline*}
Consequently, as $\P(T > \ell_t) \sim_{t \to \infty} e^{-\alpha c_t}$, we conclude that
\begin{multline*}
  \log \E\left( \exp\left( -\crochet{\mathcal{E}^\mathrm{p}_t,\phi} -\crochet{\tilde{\mathcal{E}}^\mathrm{i}_t, \psi} \right)\middle| \mathcal{F}_{c_t+A}\right)\\
  \sim_{t \to \infty} - e^{\alpha A}e^{-\alpha (c_t+A)} \sum_{u \in \mathcal{N}_{c_t+A}} h(c_t+A-b_u),
\end{multline*}
with $h(x) = \frac{1}{\P(T > x)} \left((1-e^{-\phi(x-A)}) e^{-\beta(x-A)} + \int_{-A}^{x-A} (1 - e^{-\psi(y)}) \beta e^{-\beta y} \dd y \right)$.

We now use Fact~\ref{fct:Nerman}, observing that
\[
  m^h = c_\star  \int_0^\infty h(t) \alpha e^{-\alpha  t} \P(T > t)\dd t <\infty,
\]
to deduce that
\[
  \lim_{t \to \infty} \log \E\left( \exp\left( -\crochet{\mathcal{E}^\mathrm{p}_t,\phi} -\crochet{\tilde{\mathcal{E}}^\mathrm{i}_t, \psi} \right)\middle| \mathcal{F}_{c_t+A}\right)= - e^{\alpha A} m^h Z_\infty \text{ in probability,}
\]
therefore, by the dominated convergence theorem and the tower property of conditional expectation, we have
\[
  \lim_{t \to \infty} \E\left(\exp\left( -\crochet{\mathcal{E}^\mathrm{p}_t,\phi} -\crochet{\tilde{\mathcal{E}}^\mathrm{i}_t, \psi} \right) \right)  = \E\left( e^{- e^{\alpha A}m^h Z_\infty} \right).
\]
Finally, since $\phi, \psi$ are supported on $[-A,\infty)$, we note that the  positive constant $m^h$ satisfies
\begin{align*}
  &\phantom{=}c_\star^{-1} e^{\alpha A}m^h = \int_0^\infty  h(x) \P(T > x) \alpha e^{-\alpha (x-A)} \dd x\\
  &= \int_0^\infty (1 - e^{-\phi(x-A)}) \alpha e^{-(\alpha + \beta) (x-A)} \dd x + \int_0^\infty \left\{\int_{-A}^{x-A} (1 -e^{-\psi(y)})\beta e^{-\beta y} \dd y \right\} \alpha e^{-\alpha(x-A)} \dd x\\
  &= \int_\R \alpha (1 - e^{-\phi(x)}) e^{-(\alpha + \beta) x} \dd x + \int_{\R} \int_{-\infty}^x \beta (1 - e^{-\psi(y)}) e^{-\beta y} \dd y \alpha e^{-\alpha x} \dd x\\
  &= \int_\R \left\{\alpha (1 - e^{-\phi(x)}) + \beta (1-e^{-\psi(x)})\right\} e^{-(\alpha + \beta) x} \dd x.
\end{align*}
The proof is now complete.
\end{proof}

We now turn to the proof of Theorem~\ref{thm:main}, using Proposition~\ref{prop:cv} and Lemma~\ref{lem:goodApproximation}.

\begin{proof}[Proof of Theorem~\ref{thm:main}]
By Proposition~\ref{prop:cv}, we obtain immediately that
\[
  \lim_{t \to \infty} (\mathcal{E}^\mathrm{p}_t, \tilde{\mathcal{E}}^\mathrm{i}_t) =  (\mathcal{E}^\mathrm{p}_\infty, {\mathcal{E}}^\mathrm{i}_\infty),
\]
in law for the topology of vague convergence, by identification of the Laplace transform of the limit, see \cite[Chapter 15]{Kal}. In addition, as the test functions we considered are unbounded on the right, we also obtain the convergence in distribution
\[
  \lim_{t \to \infty} (\max \mathcal{E}^\mathrm{p}_t, \max \tilde{\mathcal{E}}^\mathrm{i}_t) =  (\max \mathcal{E}^\mathrm{p}_\infty, \max {\mathcal{E}}^\mathrm{i}_\infty).
\]
Finally the proof of Theorem~\ref{thm:main} follows from Lemma~\ref{lem:goodApproximation} and Slutsky's lemma.
\end{proof}

{We end this article with the proof of Theorem~\ref{thm:ps}, that can be decomposed into two parts. We first show an upper bound on the growth rate of the longest edges, before obtaining an analogue lower bound.}

\begin{lemma}
\label{lem:ub}
Under assumptions \eqref{eqn:supercritical}, \eqref{eqn:malthusianParameter} and \eqref{eqn:regularTail}, we have
\[
  \limsup_{t \to \infty} \frac{\max\left\{M_t^\mathrm{i},M_t^\mathrm{p}\right\}}{t} \leq \frac{\alpha}{\alpha + \beta} \quad \text{a.s.}
\]
where $\alpha$ is the constant defined in \eqref{eqn:malthusianParameter} and $\beta$ the one defined in \eqref{eqn:regularTail}.
\end{lemma}

\begin{proof}
We write $M_t = \max \left\{M_t^\mathrm{i},M_t^\mathrm{p}\right\}$ the length of the longest overall branch alive at time $t$. Using that $M_t \leq t$, this result is immediate in the case $\beta = 0$, we therefore restrict ourselves in this proof to the case $\beta > 0$.

Let $\epsilon > 0$. Using that a branch of length $\ell$ as to be born before time $t - \ell$, we have by union bound:
\begin{equation*}
  \P(M_t > \ell)
  \leq \E\left( \sum_{u \in \T_{t-\ell}} \ind{T_u > \ell} \right)
  = \E(\#\T_{t-\ell}) \P(T > \ell),
\end{equation*}
in a similar way as in the proof of Lemma~\ref{lem:goodApproximation}. Using \eqref{eqn:regularTail}, we observe that for all $\beta' \in (0,\beta)$, we have
\[
  \lim_{\ell \to \infty} \P(T > \ell) e^{\beta' \ell} = 0.
\]
As a result of taking $\beta'$ sufficiently close to $\beta$ in this, and also using  \eqref{eqn:typicalGrowthRate}, we obtain that for all $\epsilon \in (0,\frac{\beta}{\alpha + \beta})$, there exists $C > 0$ such that for all $t$ large enough,
\begin{align*}
  \P(M_t > (\tfrac{\alpha}{\alpha + \beta}+\epsilon) t)
  &\leq C \e^{t \alpha (\frac{\beta}{\alpha + \beta} - \epsilon)} e^{-t\beta(\frac{\alpha}{\alpha + \beta}+\frac{\epsilon}{2})}\\
  &\leq C e^{-t \epsilon(\alpha + \frac{\beta}{2})}.
\end{align*}
Therefore, from a direct application of Borel-Cantelli lemma, we have
\[
  \limsup_{n \to \infty} \frac{M_n}{n} \leq \frac{\alpha}{\alpha + \beta} \quad \text{a.s.}
\]
To extend this result at all positive times, we remark that $t \mapsto M_t$ is a.s. non-decreasing, therefore $M_\floor{t} \leq M_t \leq M_\ceil{t}$, which completes the proof.
\end{proof}

The lower bound is obtained again with an application of the Borel-Cantelli lemma, using a classical two-steps decomposition of the branching process expressing that each inherited event occurring with positive probability will occur almost surely on the survival event of the branching process.
\begin{lemma}
\label{lem:lb}
Under the assumptions of Theorem \ref{thm:ps}, we have
\[
  liminf_{t \to \infty} \frac{M_t^\mathrm{p}}{t} \geq \frac{\alpha}{\alpha + \beta} \quad \text{a.s. on the survival event}.
\]
In addition, if $\beta > 0$ then
\[
  \liminf_{t \to \infty} \frac{M_t^\mathrm{i}}{t} \geq \frac{\alpha}{\alpha + \beta} \quad \text{a.s. on the survival event}.
\]
\end{lemma}

\begin{proof}
We only present the proof for $M_t^\mathrm{i}$ in the case $\beta > 0$, the proof for $M_t^\mathrm{p}$ following from the same computation. Using Theorem~\ref{thm:main} and that $\ell_t \sim t \frac{\alpha}{\alpha + \beta}$, for all $\delta > 0$ we have
\[
  \lim_{t \to \infty} \P(M_t^\mathrm{i} \leq t\tfrac{\alpha}{\alpha + \beta}(1 - \delta)) = \P(Z_\infty = 0) < 1,
\]
since the Sevast'yanov process survives with positive probability. Consequently, for all $\delta > 0$, there exists $\rho < 1$ such that for all $t$ large enough, we have
\begin{equation}
  \label{eqn:firstTool}
  \P(M_t^\mathrm{i} \leq t \tfrac{\alpha}{\alpha + \beta}(1 - \delta)) < \rho.
\end{equation}

We now write $\bar{N}_t = \sum_{u \in \T} \ind{b_u \in [t-1,t]}$ the number of individuals born between time $t$ and $t+1$. For $n \in \N$, we take interest in the event
\[
  A_n = \{M_n^\mathrm{i} \leq n \tfrac{\alpha}{\alpha + \beta}(1 - \delta)^2 \} \cap \{\bar{N}_{\delta n} \geq n\}.
\]
Using the branching property, i.e. the conditional independence of the $\bar{N}_{\delta n}$ subtrees born between times $n\delta - 1$ and $n \delta$, we have
\[
  \P(A_n) \leq \E\left( \rho^{\bar{N}_{\delta n}} \ind{\bar{N}_{\delta n} \geq n} \right) \leq \rho^n.
\]
Therefore, by Borel-Cantelli lemma, almost surely for $n$ large enough, we have $M_n^\mathrm{i} \geq n \tfrac{\alpha}{\alpha + \beta}(1-\delta)^2$ or $\bar{N}_{\delta n} \leq n$.

By Fact~\ref{fct:Nerman}, we have
\[
   \lim_{t \to \infty} e^{-\alpha t} \bar{N}_t = c Z_\infty \quad \text{a.s.}
\]
with $c$ a strictly positive constant. Consequently, almost surely for $n$ large enough, we have $\bar{N}_{\delta n} \geq n$ or $Z_\infty = 0$. In view of the previous result, and using that $\{Z_\infty = 0\}$ corresponds to the extinction time of the Sevast'yanov process, we conclude that a.s. on the survival event, $\liminf_{n \to \infty} \frac{M_n^\mathrm{i}}{n} \geq \frac{\alpha}{\alpha + \beta}$. This result is extended to $(M_t^\mathrm{i}, t \in \R_+)$ using the monotonicity of $t \mapsto M^\mathrm{i}_t$.
\end{proof}

\begin{proof}[Proof of Theorem~\ref{thm:ps}]
The theorem is obtain as an immediate consequence of Lemmas \ref{lem:ub} and \ref{lem:lb}.
\end{proof}

\paragraph{Acknowledgements.}
This work was carried out during the \emph{The 6th International Workshop on Branching Processes and their Applications}, we acknowledge the Universidad de Extremadura and the organisation team for their hospitality. In addition, SB is partially funded by  Research Development Fund of Xi'an Jiaotong-Liverpool University (RDF-23-01-134)., SH is partially supported by New Zealand Royal Society Te Apārangi Marsden Fund (22-UOA-052), and BM has received financial support from the CNRS through the MITI interdisciplinary program 80PRIME GEx-MBB.

In addition, we wish to thank the referee for their careful proofreading and for correcting an error in the previous version of this article, as well as Vladimir Vatutin for providing us references to \cite{Sevastyanov,Sevastyanov2}.

\bibliographystyle{plain}

\end{document}